\documentclass[english]{article}

\usepackage[english]{babel}

\usepackage[dvips]{graphicx}
\usepackage{color}
\usepackage{amssymb}
\usepackage{amsthm}
\usepackage{amsmath}
\usepackage{amscd}
\usepackage{comment}
\usepackage{datetime}
\usepackage{bm}
\usepackage{here}

 \usepackage[colorlinks=true]{hyperref}

\newtheorem{theorem}{Theorem}[section]
\newtheorem{corollary}[theorem]{Corollary}
\newtheorem{lemma}[theorem]{Lemma}
\newtheorem{proposition}[theorem]{Proposition}

\theoremstyle{definition}
\newtheorem{definition}[theorem]{Definition}

\newtheorem{remark}[theorem]{Remark}

\newcommand{\ben}{\begin{enumerate}}
\newcommand{\een}{\end{enumerate}}
\newcommand{\bit}{\begin{itemize}}
\newcommand{\eit}{\end{itemize}}
\newcommand{\be}{\begin{equation}}
\newcommand{\ee}{\end{equation}}
\newcommand{\bdm}{\begin{displaymath}}
\newcommand{\edm}{\end{displaymath}}
\newcommand{\bea}{\begin{eqnarray}}
\newcommand{\eea}{\end{eqnarray}}

\begin{document}
  %%%% Title
  \title{Asymptotic (statistical) periodicity in two-dimensional maps}
%  \subtitle{Subtitle}

  %%% Author 1
%  \author{Fumihiko Nakamura\thanks{nfumihiko@math.sci.hokudai.ac.jp}
%   \and Michael C. Mackey\thanks{michael.mackey@mcgill.ca}}

  %%%% Date
    \date{}

  \maketitle
  
  \vspace{-1cm}
  
  % Enter the first author's name and address:
\centerline{ Fumihiko Nakamura\footnote{nfumihiko@mail.kitami-it.ac.jp}}
\medskip
{\footnotesize
% please put the address of the first author
 \centerline{Kitami Institute of Technology,}
%   \centerline{Other lines}
   \centerline{165 Koen-cho, Kitami city, Hokkaido, 090-8507, Japan}
} % Do not forget to end the {\footnotesize by the sign }

\medskip

\centerline{ Michael C. Mackey\footnote{michael.mackey@mcgill.ca}}
\medskip
{\footnotesize
 % please put the address of the second  and third author
 \centerline{McGill University,}
 %  \centerline{Other lines}
   \centerline{3655 Promenade Sir William Osler, Montreal, Quebec H3G 1Y6, Canada}
}

\abstract{%it should not exceed 200 words.
In this paper we give a new sufficient condition for asymptotic periodicity of Frobenius--Perron operator corresponding to two--dimensional maps. The result of the asymptotic periodicity for strictly expanding systems, that is, all eigenvalues of the system are greater than one, in a high-dimensional dynamical systems was already known. Our new theorem enables to apply for the system having an eigenvalue smaller than one. The key idea for the proof is a function of bounded variation defined by line integration. Finally, we introduce a new two-dimensional dynamical system exhibiting the asymptotic periodicity with different periods depending on parameter values, and discuss to apply our theorem to the model.
}

\vspace{\baselineskip}

2020 Mathematics Subject Classification. Primary: 37A30, 26A45; Secondary: 37E30.

Key words and phrases. Asymptotic periodicity,  bounded variation, Frobenius--Perron operator, two-dimensional map, Farey series.
  %%%% Body

\section{Introduction}In examining the behaviour of dynamical systems, two main complementary threads have emerged.   In one, the evolution of trajectories is the main focus, while in the other the evolution of densities is considered.  In the latter case, one can think of the evolution of densities representing the overall statistical behaviour when a large (`infinite') number of trajectories are examined.  In this paper we focus on the second point of view, which is closely related to early work in statistical physics initiated by both Boltzmann \cite{boltzmann96} and Gibbs \cite{gibbs02} over a century ago and which forms the basis of the field of ergodic theory.

In examining the evolution of densities, there are three major types of behaviour that may occur and they are ergodicity, mixing, and exactness \cite{almcmbk94}.  In addition there is a less well known  fourth type of behaviour,  called asymptotic periodicity (or statistical periodicity), which was first introduced and studied
by Keller \cite{keller1982ergodic}.  We will say more about these four types of behaviour in Section \ref{sec:back}.

Asymptotic periodicity %, first introduced by \cite{keller1982ergodic},
is known to occur in  deterministic discrete time dynamical systems \cite{lasota-li-yorke,komornik1986asymptotic,lasota1986statistical,provatas91a} as well as being  induced by noise \cite{almcm87,provatas1991noise,nakamura2017}.
One example of asymptotic periodicity in a deterministic setting is that of the hat (or tent) map
\begin{equation}
    x_{n+1}=\left\{
    \begin{array}{ll}
    a x_n   & x_n \in [0,\frac 12]\\
    a(1-x_n) & x_n \in (\frac 12, 1],
    \end{array}
    \right.
    \label{eq:hat map}
    \end{equation}
which was considered by Ito \cite{ito79a,ito79b}, Shigematsu \cite{shig83}, and Yoshida \cite{yoshida83} initially and then by Provatas \cite{provatas91a} within the framework of asymptotic periodicity. To our knowledge the only studies of noise induced asymptotic periodicity are in the noise perturbed Nagumo-Sato \cite{nagumo1972} map (also known as the  Keener \cite{keener1980chaotic} map)  and given by
\begin{equation}
x_{n+1} = \alpha x_n + \beta + \xi_n \qquad \mbox{mod}\ 1 \label{eq:NS}
\end{equation}
where $0 < \alpha, \beta < 1 $ and the $\{\xi_n\}$ are independent random variables distributed with a density $g$, and studied by \cite{almcm87,provatas1991noise,nakamura2017}.

 In this paper we present a new theorem on asymptotic periodicity in maps of dimension greater than one, extending the result of \cite{gora1989absolutely} for asymptotic periodicity in a high-dimensional dynamical systems which was stated for strictly expanding systems, that is, for systems in which all eigenvalues are greater than one.

In Section \ref{ssec:tools}  we summarize some elementary concepts and tools from ergodic theory, and then in Section \ref{ssec:BV} give some background and simple results on bounded variation for functions of two variables that will be essential in the proof of our main Theorem \ref{maintheorem} in Section \ref{sec:main}.  In Section \ref{sec:example} we consider an example of our main theorem  and illustrate how the period changes as parameters are changed.

\section{Background}\label{sec:back}

\subsection{Tools and definitions from ergodic theory}\label{ssec:tools}  This section collects together some basic concepts needed later.  Consult \cite{almcmbk94} for more details.

Let $(X,\mathcal A,\mu)$ be a measure space and assume that a system  has states
distributed in a phase  space $X$, and that the distribution of
these states is characterized by a time dependent density
$f_n(x)$, $n\in\mathbb{N}$.
Remember that $f$ is a {\bf density} if
$f(x) \geq 0,\ \int_{X} f(x)\,d\mu(x) = 1$.
Equilibrium is characterized
by a  time independent density $ f_*(x)$.
Given a phase space $X$ we will  denote the space of all
densities  on $X$ by $D(X)$ or by $D$
if $X$ is understood.

Also think of a dynamics $S$ operating on the same phase space $X$,
$S:X\rightarrow X$.
One way to think about  a dynamics
is through the evolution of a trajectory emanating from a single initial
condition in the
phase space $X$, and a complementary approach is to  study
how a density of initial conditions evolves under the action of the dynamics.  With a dynamics $S$ and initial density $f_0(x)$ of
states, the evolution of the density $f_n(x)$ is given by $f_n(x) = P_S^nf_0(x)$,
wherein $P_S$ is the Markov (or evolution transfer) operator corresponding to $S$.

\begin{definition}
Any operator $P:L^1(X) \rightarrow L^1(X)$ that satisfies
$$
Pf \geq 0\quad \text{and} \quad \lVert Pf \rVert_{L^1} = \lVert f \rVert_{L^1}
$$
for any $f \geq 0$, $f \in  L^1(X)$ is called a {\bf Markov (or evolution)
operator}.  {\it If we restrict ourselves to only considering densities
$f$, then any operator $P$ which when acting on a density again yields a
density is a density evolution operator.}
\end{definition}
\noindent Given an evolution operator $P$ operating on densities alone, so $P:D \to D$, if there is a density $f_*$ such that $Pf_* = f_*$ then  $f_*$ is called a {\bf stationary density}.

\begin{definition}Let $(X, {\mathcal A}, \mu)$ be a measure space .
If $S $ is a nonsingular transformation, then the unique Markov operator
$P:L^1(X) \rightarrow L^1(X) $ defined by
\be
\int_A P f(x)\, d\mu(x) =  \int_{S^{-1}(A)}
f(x) \,d\mu(x)
\label{4.1}
\ee
is called the {\bf Frobenius-Perron operator} corresponding to $S$.
\end{definition}

\begin{definition}Let $(X, {\mathcal A}, \mu)$ be a measure space and let a nonsingular transformation $S\colon X \rightarrow X$ be given. Then $S$ is called {\bf ergodic} if every invariant set $A \in {\mathcal A}$ (i.e. $S^{-1}(A)=A$) is such that either $\mu(A) = 0$ or $\mu (X\setminus A) = 0$; that is, $S$ is ergodic if all invariant sets are {\bf trivial} subsets of $X$.
\end{definition}
\noindent Ergodicity is equivalent to:

 \begin{theorem}\cite[Theorem 4.4.1a]{almcmbk94}
 Let $(X, {\mathcal A}, \mu)$ be a normalized measure space, $\mu(X)=1$.
  A dynamics $S$ on a phase space $X$ with
Frobenius-Perron operator $P_S$ and unique stationary density $f_*$ is {\it ergodic} if and only if $\lbrace P^nf_0 \rbrace $ is Ces\`aro convergent to $f_*$ for all initial densities $f_0$, {\it i.e.}, if
\begin{equation}
\lim_{n\to \infty } \frac{1}{n} \sum_{k=0}^{n-1} <P^kf_0,g>  = <f_*,g>
\label{e-ergod}
%\fbox{e-ergod}
\end{equation}
where $g$ is any bounded measurable function and
\begin{equation}
<f,g> = \int_X f(x)g(x)d\mu(x)
\end{equation}
denotes the $\mathbb{R}$-valued inner product.
\end{theorem}

\begin{definition} Let $(X, {\mathcal A}, \mu)$ be a normalized measure space,
and  $S\colon X \rightarrow X$ a measure-preserving transformation. $S$ is called {\bf mixing} if
\begin{equation}
\lim_{n\rightarrow \infty} \mu (A \cap S^{-n}(B)) = \mu (A) \mu (B) \qquad \mbox{for all } A,B \in {\mathcal A}.
\end{equation}
\end{definition}
\noindent Mixing implies ergodicity and is equivalent to:

%\cite{almcmbk94} have shown that the following is equivalent to mixing.

 \begin{theorem}\cite[Theorem 4.4.1b]{almcmbk94}
  A dynamics $S$ on a phase space $X$ with
Frobenius-Perron operator $P_S$ and unique stationary density $f_*$ is {\it mixing} if and only if
\begin{equation}
\lim_{n \to \infty} < P_S^nf_0,g> = <f_*,g>,
\label{e-mix}
%\fbox{e-mix}
\end{equation}
for every initial density $f_0\in\mathcal{D}$ and bounded measurable function $g$.
\end{theorem}

\begin{definition}  Let $(X, {\mathcal A}, \mu)$ be a normalized measure space and $S\colon X \rightarrow X$ a measure-preserving transformation such that $S(A) \in {\mathcal A}$ for each $A \in {\mathcal A}$. If
\begin{equation}
\lim_{n\rightarrow \infty} \mu(S^{n}(A)) = 1\qquad \mbox{for every } A \in {\mathcal A}, \mu(A) > 0,
\end{equation}
then $S$ is called {\bf exact} or {\bf asymptotically stable}.
\end{definition}

\noindent Exactness implies mixing and is equivalent to:

%This is equivalent \cite{almcmbk94} to the following.
 \begin{theorem}\cite[Theorem 4.4.1c]{almcmbk94}
A dynamics $S$ on a phase space $X$ with
Frobenius-Perron operator $P_S$ and unique stationary density $f_*$ is {\it asymptotically stable } if and only if
\begin{equation}
\lim_{n \to \infty} || P_S^nf - f_*||_{L^1} = 0
\label{e-exact}
%\fbox{e-exact}
\end{equation}
for {\it every} initial density $f \in D$.
\end{theorem}
Asymptotically stable systems have a number of interesting properties (cf. \cite{almcmbk94,mcmtdbk} for more complete details).  Asymptotically stable systems are non-invertible and they always have a unique stationary  density $f_*$.

Next, we define a smoothing Markov operator.
\begin{definition} Let $(X,\mathcal A,\mu)$ be a measure space.
A Markov operator $P$ is said to be {\bf smoothing} (or {\bf constrictive})  if there exists a set
$A$ of finite measure, and two positive
constants $k < 1$ and $\delta   > 0$ such
that for every set $E$ with $\mu(E) < \delta$  and every density $f$
there is some integer $n_0(f,E)$ for which
$$
\int_{E \cup (X \setminus A)} P^nf(x)\, d\mu(x) \leq k \qquad \text{for  }n \geq
n_0(f,E).
$$
\end{definition}
This definition of smoothing just means that any initial
density, even if concentrated on a small region of the phase space $X$,
will eventually be 'smoothed' out by $P^n$ and not end up looking like a delta function.  Notice that if
$X$ is a finite phase space we can take $X=A$ so the smoothing condition
looks simpler:
$$
\int_{E} P^nf(x)\, d\mu(x) \leq k \qquad \text{for  }n \geq
n_0(f,E).
$$

Smoothing
operators are important because of a theorem of \cite{komornik87} introduced next,
first proved in a more restricted situation by \cite{lasota-li-yorke}. Although the property called weakly constrictive introduced in \cite{lasota-li-yorke} and \cite{komornik87} seems to be different from  smoothing, it also leads to asymptotic periodicity. Conversely, we can immediately show that an asymptotically periodic Markov operator is smoothing and weakly constrictive in the sense of \cite{lasota-li-yorke}. Thus we conclude smoothing and weakly constrictiveness  are equivalent.

\begin{theorem} [Spectral Decomposition Theorem, \cite{komornik87}\label{thm:spec-decomp}] Let $P$ be a smoothing Markov operator.  Then there is an integer
$r > 0$, a sequence of nonnegative densities $g_{i} $ and a sequence of bounded
linear functionals   $\lambda _i$, $i = 1,\ldots ,r, $ and an operator $Q:L^1(X)\rightarrow L^1(X) $ such that for all densities $f$, $Pf $ has the form
\be
Pf(x) =  \sum_{i=1}^r \lambda _{i}(f)g_{i}(x) + Qf(x).
\label{6.1}
\ee
The densities $g_{i} $ and the transient
operator $Q$ have the following properties:
\ben
\item      The $g_{i}$ have disjoint support (i.e. are mutually orthogonal
and thus form a basis set),
so $g_{i}(x)g_j(x) = 0 $ for all $i \ne j$.
\item      For each integer $i$ there is a unique integer $\alpha   (i) $ such
that $Pg_{i} = g_{\alpha (i)}$.  Furthermore, $\alpha (i) \ne  \alpha (j) $ for $i
\ne j$.  Thus the operator $P$ permutes the densities $g_{i}$.
\item $\parallel P^nQf \parallel \rightarrow 0 $ as $n \rightarrow \infty, \qquad n \in \mathbb{N}$.
\een
\end{theorem}
Notice from \eqref{6.1} that $P^{n+1}f$ may be immediately written in the
form
\be
P^{n+1}f(x) =  \sum_{i=1}^r \lambda _i (f)g_{\alpha ^t (i)}(x) + Q_{n}f(x), \qquad n \in \mathbb{N}
\label{6.2}
\ee
where $Q_{n} = P^{n}Q$, $\parallel Q_{n}f \parallel \rightarrow 0$ as $n
\rightarrow \infty$, and  $\alpha ^n (i) = \alpha (\alpha ^{n-1}(i)) =
\cdots$. The density terms in the summation of \eqref{6.2} are just permuted by each
application of $P$.  Since $r$ is finite, the series
\be
\sum _{i=1}^r \lambda _i (f) g_{\alpha ^t (i)}(x)
\label{6.3}
\ee
must be periodic with a period $T \leq r!$.  Further, as
$
\lbrace \alpha ^n(1),\ldots, \alpha ^n(r) \rbrace
$
is just a permutation of ${1,\cdots ,r}$ the summation \eqref{6.3} may be
written in the alternative form
$$
\sum _{i=1}^r \lambda _{\alpha^{-t}(i)}(f)g_i(x),
$$
where  $\alpha^{-n}(i)$ is the inverse permutation  of $\alpha ^n(i)$.

     This rewriting of the summation portion of \eqref{6.2} makes the effect of
successive applications of $P$ completely transparent.  Each application of
$P$ simply
permutes the set of scaling coefficients associated with the densities
$g_{i}(x)$ [remember that these densities have disjoint support].

 Since $T$ is finite and the summation \eqref{6.3} is periodic (with a
period bounded above by $r!$), and $\parallel Q_{n}f \parallel \rightarrow
0$ as $n \rightarrow \infty$, we say that for any smoothing Markov operator
the sequence $\lbrace P^nf \rbrace$ is {\bf asymptotically (statistically) periodic} or,
more briefly, that $P$ is {\bf asymptotically periodic}. Komorn\'{i}k \cite{komornik91} has reviewed the subject of asymptotic periodicity.

     Asymptotically periodic Markov operators always have at least one
stationary density given by
\be
f_*(x) = \dfrac{1}{r} \sum _{i=1}^r g_i(x),
\label{ap-sd}
\ee
where $r$ and the $g_{i}(x)$ are defined in Theorem \ref{thm:spec-decomp}.
It is easy
to see that $f_*(x)$ is a stationary density, since by Property 2 of
Theorem \ref{thm:spec-decomp} we also have
$$
Pf_*(x) = \dfrac {1}{r} \sum_{i=1}^r g_{\alpha (i)}(x),
$$
and thus $f_*$ is a stationary density of $P^n$.  Hence, for any smoothing
Markov operator the stationary density \eqref{ap-sd} is just the average of the
densities $g_{i}$.

\begin{remark}
 It is known \cite[Section 5.5]{almcmbk94} that mixing, exactness and asymptotically periodicity with $r=1$ are all equivalent for a smoothing Markov operator. This means that the case $r=1$ has a strictly stronger mixing property than the case $r>1$. In terms of published examples having periodicity with not only $r=1$ but also $r>1$, we only know the hat map \eqref{eq:hat map} and the noise perturbed \cite{nagumo1972}  map \eqref{eq:NS} (see section \ref{sec:example} for a discussion of the parameters of the hat map showing asymptotic periodicity when $r>1$). The model we introduce in Section \ref{sec:example} is a new two-dimensional example having different periods depending on parameter values.
 \end{remark}

\subsection{Functions of bounded variation in two variables}\label{ssec:BV}

There are many definitions of the total variation for functions of two real variables. For example, see \cite{vitali1904sulle} and \cite{hardy1905double} summarized in \cite{adams1934properties,chistyakov2010maps}. In this paper, we refer to the definition in \cite{ashton2005functions} which is defined using line integration.

Consider a compact subset $\sigma\subset\mathbb{R}^2$, a function $f:\sigma\to\mathbb{R}$ and a continuous and piecewise $C^1$ curve $\gamma:[0,1]\to\mathbb{R}^2$. Although Ashton \cite{ashton2005functions} found it sufficient to consider polygonal curves, that is, piecewise linear continuous curves, we need to treat more general continuous curves since we focus on non-linear transformations. We denote the set of all continuous and piecewise $C^1$ curves by $\Gamma$.

\begin{definition}
Let $\gamma\in\Gamma$, then $\{(x_i,y_i)\}_{i=1}^n$ is called a {\bf partition} of $\gamma$ over $\sigma$ if $(x_i,y_i)\in\sigma$ for all $i$ and there exists a partition $\{s_i\}_{i=1}^n \in \Lambda([0,1])$ such that $(x_i,y_i)=\gamma(s_i)$ for all $i$, where  $\Lambda([0,1])$ is the set of all partitions of $[0,1]$. The set of all partitions of $\gamma$ over $\sigma$ is denoted by $\Lambda(\gamma,\sigma)$.
\end{definition}

\begin{definition}
Let $\sigma\subset\mathbb{R}^2$ be  compact, and consider a function $f:\sigma\to\mathbb{R}$ and a curve $\gamma\in\Gamma$. The {\bf variation} of $f$ along the curve $\gamma$ is defined as
\begin{eqnarray}
{\rm cvar}(f,\gamma,\sigma):=\sup_{\{(x_j,y_j)\}_{j=1}^{n} \in\Lambda(\gamma,\sigma)}\sum_{j=1}^{n-1}|f(x_{j+1},y_{j+1})-f(x_j,y_j)|.
\end{eqnarray}
\end{definition}

\begin{remark}
From the definition, one can rewrite ${\rm cvar}(f,\gamma,\sigma)$ as
\begin{eqnarray}
{\rm cvar}(f,\gamma,\sigma)=\sup_{\substack{\{t_j\}_{j=1}^{n} \in\Lambda([0,1]) \\ \gamma(t_j)\in\sigma}}\sum_{j=1}^{n-1}|f\circ\gamma(t_{j+1})-f\circ\gamma(t_j)|.
\end{eqnarray}
Note that we sometimes omit $\gamma(t_j)\in\sigma$ and simply write $\displaystyle\sup_{\{t_j\}_{j=1}^{n} \in\Lambda([0,1])}$ for the above equation.
\end{remark}

The following basic properties for the variation are known.

\begin{proposition}\label{procvar}
(\cite[Proposition 3.2]{ashton2005functions})
Let $\sigma_1\subset\sigma$ be a nonempty compact subset of $\mathbb{R}^2$, $f,g:\sigma\to\mathbb{R}$, $\gamma\in\Gamma$ and $\alpha\in\mathbb{R}$. Suppose $\gamma=\gamma_1\circ\gamma_2\in\Gamma$ with $\gamma_1(1)\in\sigma$. Then,
\begin{itemize}
\item[(i)] ${\rm cvar}(f+g,\gamma,\sigma)\leq {\rm cvar}(f,\gamma,\sigma) + {\rm cvar}(g,\gamma,\sigma)$,
\item[(ii)] ${\rm cvar}(fg,\gamma,\sigma)\leq \|f\|_{\infty}{\rm cvar}(g,\gamma,\sigma)+\|g\|_{\infty}{\rm cvar}(f,\gamma,\sigma)$,
\item[(iii)] ${\rm cvar}(\alpha f,\gamma,\sigma)=|\alpha|{\rm cvar}(f,\gamma,\sigma)$,
\item[(iv)] ${\rm cvar}(g,\gamma,\sigma)={\rm cvar}(g,\gamma_1,\sigma)+{\rm cvar}(g,\gamma_2,\sigma)$,
\item[(v)] ${\rm cvar}(g,\gamma_1,\sigma)\leq{\rm cvar}(g,\gamma,\sigma)$,
\item[(vi)] ${\rm cvar}(g,\gamma,\sigma_1)\leq{\rm cvar}(g,\gamma,\sigma)$.
\end{itemize}
\end{proposition}

\begin{definition}
The compact and connected sets $\sigma_1,\sigma_2$ are said to be {\bf adjacent} if $\sigma_1\cap\sigma_2\neq\emptyset$ and $int(\sigma_1\cap\sigma_2)=\emptyset$.
\end{definition}

Now we note the following property for the ${\rm cvar}(f,\gamma,\sigma)$.

\begin{proposition}\label{procvar2}
(\cite[Theorem 4.9]{gimenez2014functions})
Let $\sigma_1,\sigma_2$ be two compact and connected adjacent sets. Then, for any $f:\sigma_1\cup\sigma_2\to\mathbb{R}$,
\begin{eqnarray}
{\rm cvar}(f,\gamma,\sigma_1\cup\sigma_2)={\rm cvar}(f,\gamma,\sigma_1)+{\rm cvar}(f,\gamma,\sigma_2).\nonumber
\end{eqnarray}
\end{proposition}

\begin{lemma}\label{lemcvar2}
Let $\sigma_1\subset \sigma$ be a nonempty compact on $\mathbb{R}^2$, $f:\sigma\to\mathbb{R}$ and $\gamma\in\Gamma$. Assume that $g:\sigma_1\to g(\sigma_1)$ is a one-to-one map. Then,
$$
{\rm cvar}(f\circ g,\gamma, \sigma_1) = {\rm cvar}(f,g\circ\gamma, g(\sigma_1))
$$
\end{lemma}

\begin{proof}
\begin{eqnarray}
{\rm cvar}(f\circ g,\gamma,\sigma_1)
&=& \sup_{\substack{\{t_j\}_{j=1}^{n} \in\Lambda([0,1]) \\ \gamma(t_j)\in\sigma_1}}\sum_{j=1}^{n-1}|f\circ g(\gamma(t_{j+1}))-f\circ g(\gamma(t_j))|\nonumber\\
&=& \sup_{\substack{\{t_j\}_{j=1}^{n} \in\Lambda([0,1]) \\ g\circ\gamma(t_j)\in g(\sigma_1)}}\sum_{j=1}^{n-1}|f( g\circ\gamma(t_{j+1}))-f(g\circ\gamma(t_j))|\nonumber\\
&=& {\rm cvar}(f,g\circ\gamma, g(\sigma_1))\nonumber
\end{eqnarray}
\end{proof}

\begin{definition}
Let $\mathcal{C}$ be the set of all convex closed Jordan curve on $\mathbb{R}^2$. Then $t\in[0,1]$ is said to be an {\bf entry point} of $\gamma\in\Gamma$ on a curve $c\in\mathcal{C}$ if either
\begin{itemize}
\item[(i)] $t=0$ and $\gamma(0)\in c$, or
\item[(ii)] $\gamma(t)\in c$ and for all $\varepsilon >0$ there exists $s\in (t-\varepsilon,t)\cap [0,1]$ such that $\gamma(s)\notin c$.
\end{itemize}
Set ${\rm vf}(\gamma,c)$ to be the number of entry points of $\gamma$ on $c\in\mathcal{C}$ and ${\rm vf}(\gamma)$ to be the supremum of ${\rm vf}(\gamma,c)$ over all convex closed Jordan curves $c$, that is,
\begin{eqnarray}
{\rm vf}(\gamma):=\sup_{c\in\mathcal{C}}{\rm vf}(\gamma,c).
\end{eqnarray}
\end{definition}
\begin{remark}
In \cite{ashton2005functions}, ${\rm vf}(\gamma,c)$ is defined by lines instead of curves, but we need the definition by curves for our main theorem.
\end{remark}

\begin{definition}
Let $f:\sigma\to\mathbb{R}$. The {\bf variation} of $f$ on $\sigma$ is defined by
\begin{eqnarray}
{\rm Var}(f,\sigma):=\sup_{\gamma\in\Gamma}\frac{{\rm cvar}(f,\gamma,\sigma)}{{\rm vf}(\gamma)}.
\end{eqnarray}
If $\gamma\in\Gamma$ satisfies ${\rm vf}(\gamma)=\infty$ and ${\rm cvar}(f,\gamma,\sigma)=\infty$, then we define ${\rm cvar}(f,\gamma,\sigma)/{\rm vf}(\gamma)=0$.
\end{definition}

The following properties for the variation define above are well-known.

\begin{proposition}\label{provar}
Let $\sigma_1\subset\sigma$ be a nonempty compact subset of $\mathbb{R}^2$, $f,g:\sigma\to\mathbb{R}$ and $\alpha\in\mathbb{R}$. Then,
\begin{itemize}
\item[(i)] ${\rm Var}(f+g,\sigma)\leq {\rm Var}(f,\sigma) + {\rm Var}(g,\sigma)$,
\item[(ii)] ${\rm Var}(fg,\sigma)\leq \|f\|_{\infty}{\rm Var}(g,\sigma)+\|g\|_{\infty}{\rm Var}(f,\sigma)$,
\item[(iii)] ${\rm Var}(\alpha f,\sigma)=|\alpha|{\rm Var}(f,\sigma)$,
\item[(iv)] ${\rm Var}(f,\sigma_1)\leq{\rm Var}(f,\sigma)$.
\end{itemize}
\end{proposition}

\begin{proof}
The proof of all properties follows immediately from Proposition \ref{procvar}. See also \cite{ashton2005functions}.
\end{proof}

Finally, we state and prove  the following lemma in order to prove our main theorem.

\begin{lemma}\label{lemvar3}
Let $\sigma\subset\mathbb{R}^2$ be a compact set. Assume $g:\mathbb{R}^2\to\mathbb{R}$ is a $C^1$ function. If there exists a constant $C>0$ such that $|g_x(x,y)|\leq C$ and $|g_y(x,y)|\leq C$ for any $(x,y\in Int(\sigma))$, then
${\rm Var}(g,\sigma)$ is bounded.
\end{lemma}

\begin{proof}
\begin{eqnarray}
\frac{{\rm cvar}(g,\gamma,\sigma)}{{\rm vf}(\gamma)}&=&\frac{1}{{\rm vf}(\gamma)}\sup_{\{(x_i,y_i)\}_{j=1}^{n} \in\Lambda(\gamma,\sigma)}\sum_{j=1}^{n-1}|g((x_{j+1},y_{j+1})-g(x_j,y_j)|\nonumber\\
&=&\frac{1}{{\rm vf}(\gamma)}\sup_{\{(t_i)\}_{j=1}^{n} \in\Lambda([0,1])}\sum_{j=1}^{n-1}|g\circ\gamma(t_{j+1})-g\circ\gamma(t_j)|.\nonumber
\end{eqnarray}
Since $g$ is a $C^1$ function and $\gamma$ is a piecewise $C^1$ curve, then we have
\begin{eqnarray}
&\leq&\frac{1}{{\rm vf}(\gamma)}\int_{0}^{1}|(g\circ\gamma)'(t)|dt\nonumber\\
&=&\frac{1}{{\rm vf}(\gamma)}\int_{\gamma}|(g_x(x,y)\frac{dx}{dt}+g_y(x,y)\frac{dy}{dt})|dt\nonumber\\
&=&\frac{1}{{\rm vf}(\gamma)}\int_{\gamma}|g_x(x,y)||dx|+|g_y(x,y)||dy|\nonumber\\
&\leq&\frac{C}{{\rm vf}(\gamma)}\int_{\gamma}(|dx|+|dy|)\nonumber\\
&\leq&\frac{C}{{\rm vf}(\gamma)}\int_{0}^1|\gamma'(t)||dt|\nonumber\\
&\leq& C{\rm Var}(x+y,\sigma),\nonumber
\end{eqnarray}
which is bounded.  Thus ${\rm Var}(g,\sigma)$ is bounded.
\end{proof}

\section{Main theorem}\label{sec:main}

Gora and Boyarsky \cite{gora1989absolutely} gave a sufficient condition for asymptotic (statistical) periodicity in piecewise $C^2$ maps on $\mathbb{R}^N$ using a general definition of the total variation. Their assumptions are stronger than ours since they assume that the map is expanding in all directions and thus  all eigenvalues of the Jacobian are larger than one.

Our main result gives a sufficient condition for asymptotic periodicity of more general piecewise $C^2$ maps on $\mathbb{R}^2$, that are not necessarily expanding in all directions,  by using the definition of variation constructed by line integration as introduced in \cite{gimenez2014functions}. Let $X \subset \mathbb{R}^2$ be a connected compact subset.

\begin{theorem}\label{maintheorem}
Let $S:X\to X$ satisfy the following conditions:
\begin{itemize}
\item[(i)] There is a partition $I_1,I_2,\cdots,I_r$ of $X$ such that for each $i=1,\cdots,r$,
\begin{itemize}
\item the restricted map $S|_{int(I_i)}$ is a $C^2$ and one-to-one function,
\item each boundary $\partial (I_i)$ is a piecewise $C^2$ curve having a finite boundary length,
\item the set $S(I_i)$ is convex;
\end{itemize}
\item[(ii)] For $i=1,\cdots,r$, each Jacobian $J_i(x,y)$ of $S|_{int(I_i)}$ satisfies
\begin{eqnarray}
J_i(x,y)\geq\lambda>1\ \ \ {\it for}\ \ (x,y)\in int(I_i);\nonumber
\end{eqnarray}
\item[(iii)] There are real constants $C'>0$ such that, for $i=1,\cdots,r$,
\begin{eqnarray}
\left|\frac{\partial}{\partial x}J_i^{-1}(x,y)\right|\leq C'<\infty,\ \ \ {\it for}\ \ (x,y)\in int(I_i),\nonumber\\
\left|\frac{\partial}{\partial y}J_i^{-1}(x,y)\right|\leq C'<\infty,\ \ \ {\it for}\ \ (x,y)\in int(I_i);\nonumber
\end{eqnarray}
\item[(iv)] There exists $C>0$ such that for any curves $\gamma$ on $X$, a curve $\widetilde{\gamma}$ constructed by connecting all curves $\{S|_{int(I_i)}^{-1}(\gamma)\}_{i=1}^r$ whose length is minimal satisfies
\begin{eqnarray}
\sup_{\gamma\in\Gamma}\frac{{\rm vf}(\widetilde{\gamma})}{{\rm vf}(\gamma)}\leq C;\nonumber
\end{eqnarray}
\item[(v)] The numbers $\lambda,C$  satisfy
$$
\frac{C}{\lambda}<1.
$$
\end{itemize}
Let $P$ be the Frobenius-Perron operator corresponding to $S$. Then, for all $f\in D(X)$, $\{P^nf\}$ is asymptotically periodic.
\end{theorem}

\begin{remark}\label{easyex}
Item (ii) implies an area expanding property.
If the system satisfied only condition (ii), we can immediately find a counterexample of non-asymptotically periodic transformations. For example, the piecewise linear map $S(x,y)=(4x,y/2)$ mod 1 has Jacobian $\lambda=2$ but has eigenvalues $4$ and $1/2$. It is clear that the map has no absolutely continuous invariant measure with respect to Lebesgue measure, which means that the corresponding Frobenius-Perron operator is not asymptotically periodic. However, if we take a partition satisfying (i) and the system satisfies (iv) and (v), then such counterexamples can be excluded. Indeed, we find that the factor $\frac{{\rm vf}(\widetilde{\gamma})}{{\rm vf}(\gamma)}$ must be larger than $\lambda$ for the map $S$, and  (v) cannot hold.
\end{remark}

\subsection*{Proof of Theorem \ref{maintheorem}}
First we write the Frobenius-Perron operator $P$ corresponding to $S$ as
 \begin{eqnarray}
Pf(x,y)=\sum_{i=1}^{r}\rho_i(x,y)f(g_i(x,y))1_{I'_i}(x,y),\nonumber
\end{eqnarray}
where $g_i(x,y)=S_i^{-1}(x,y)$ and $\rho_i(x,y)=J_i^{-1}(x,y)$ for $(x,y)\in I'_i$ with $I'_i=S(I_i)$ and $i=1,\cdots,r$. Each $J_i(x,y)$ is a Jacobian on $I'_i$.

We then calculate the variation ${\rm Var}(Pf,X)$ for $f\in D(X)$ of bounded variation, denoted by $\underset{X}{\rm Var}(Pf)$. We first calculate, by (i) of Proposition  \ref{procvar},
\begin{eqnarray}
\underset{X}{\rm Var}(Pf)
&=& \sup_{\gamma\in\Gamma}\frac{1}{{\rm vf}(\gamma)}{\rm cvar}\left(\sum_{i=1}^{r}\rho_i(x,y)f(g_i(x,y))1_{I'_i}(x,y),\gamma,X\right)\nonumber\\
&\leq& \sup_{\gamma\in\Gamma}\frac{1}{{\rm vf}(\gamma)}\sum_{i=1}^{r}{\rm cvar}\left(\rho_i(x,y)f(g_i(x,y))1_{I'_i}(x,y),\gamma,X\right)\nonumber
\end{eqnarray}
By (ii) of Proposition \ref{procvar},
\begin{eqnarray}
{\rm cvar}\left(\rho_i(x,y)f(g_i)\cdot 1_{I'_i},\gamma,X\right)
&\leq& (\sup_{I'_i}\rho_i){\rm cvar}\left(f\circ g_i\cdot 1_{I'_i},\gamma,X\right)+{\rm cvar}\left(\rho_i,\gamma,X\right)\sup_{I'_i}f(g_i)
\nonumber\\
&\leq& \frac{1}{\lambda}{\rm cvar}\left(f\circ g_i\cdot 1_{I'_i},\gamma,X\right)+{\rm cvar}\left(\rho_i,\gamma,X\right)\sup_{I'_i}f(g_i).\nonumber
\end{eqnarray}
By the mean value theorem for definite integrals, we have
\begin{eqnarray}
\sup_{I'_i}f(g_i)\leq\frac{1}{\iint_{I'_i}dxdy}\iint_{I'_i}|f(g_i(x,y))|dxdy.\label{eq6-1}
\end{eqnarray}
Then we have
\begin{eqnarray}
\underset{X}{\rm Var}(Pf) &\leq& \sup_{\gamma\in\Gamma}\frac{1}{{\rm vf}(\gamma)}\left\{\frac{1}{\lambda}\sum_{i=1}^{r}{\rm cvar}(f\circ g_i\cdot 1_{I'_i},\gamma,X)
+\sum_{i=1}^{r}\frac{{\rm cvar}\left(\rho_i,\gamma,X\right)}{\iint_{I'_i}dxdy}\iint_{I'_i}|f(g_i(x,y))|dxdy
\right\}.\nonumber\\
&\leq& \frac{1}{\lambda}\sup_{\gamma\in\Gamma}\frac{1}{{\rm vf}(\gamma)}\sum_{i=1}^{r}{\rm cvar}(f\circ g_i\cdot 1_{I'_i},\gamma,X)\label{eq3-1}\\
&&+\sup_{\gamma\in\Gamma}\frac{1}{{\rm vf}(\gamma)}\sum_{i=1}^{r}\frac{{\rm cvar}\left(\rho_i,\gamma,X\right)}{\iint_{I'_i}dxdy}\iint_{I'_i}|f(g_i(x,y))|dxdy
.\label{eq3-2}
\end{eqnarray}
Since $\underset{X}{\rm Var}(\rho_i)$ is bounded by Lemma \ref{lemvar3}, there is some constant $\hat{C}$ such that
\begin{eqnarray}
\eqref{eq3-2}\leq
\sum_{i=1}^{r}\frac{\hat{C}}{\iint_{I'_i}dxdy}\iint_{I'_i}|f(g_i(x,y))|dxdy
.\nonumber
\end{eqnarray}
Changing the variables by $g_i(x,y)=(\hat{x},\hat{y})$,
\begin{eqnarray}
\leq\sum_{i=1}^{r}\frac{\hat{C}}{\iint_{I'_i}dxdy}\iint_{I_i}f(\hat{x},\hat{y})d\hat{x}d\hat{y}
\leq\max_i\frac{r\hat{C}}{\iint_{I'_i}dxdy} \label{eq6-2}
\end{eqnarray}
since $f\in D(X)$.
We next calculate Eq.\eqref{eq3-1}.
For $i=1,\cdots,r$, $\{\gamma(t_j)\}_{j=0}^{n-1}\subset X$, the sets $A_i,B_i,C_i$ are defined by
\begin{eqnarray}
A_i&:=&\{j=0,\cdots,n-1 \ :\ \gamma(t_j)\in I_i'\quad{\rm and}\quad\gamma(t_{j+1})\in I_i'\},\nonumber\\
B_i&:=&\{j=0,\cdots,n-1 \ :\ {\rm either}\quad \gamma(t_j)\notin I_i'\quad{\rm or}\quad\gamma(t_{j+1})\notin I_i'\},\nonumber\\
C_i&:=&\{j=0,\cdots,n-1 \ :\ \gamma(t_j)\notin I_i'\quad{\rm and}\quad\gamma(t_{j+1})\notin I_i'\}.\nonumber
\end{eqnarray}
Then,
\begin{eqnarray}
{\rm cvar}(f\circ g_i\cdot 1_{I'_i},\gamma,X)&\leq&
\sum_{j\in A_i}| f\circ g_i(\gamma(t_{j+1})-f\circ g_i(\gamma(t_{j})|\label{eq4-1}\\
&&+\sum_{j\in B_I}|\max\{f\circ g_i(\gamma(t_{j+1}),f\circ g_i(\gamma(t_{j})\}|.\label{eq4-2}
\end{eqnarray}
By definition, we have
$$
\eqref{eq4-1}\leq {\rm cvar}(f\circ g_i,\gamma,I'_i),
$$
and
$$
\eqref{eq3-1}\leq \frac{1}{\lambda}\sup_{\gamma\in\Gamma}\frac{1}{{\rm vf}(\gamma)}\sum_{i=1}^{r}{\rm cvar}(f\circ g_i,\gamma,I'_i) +\frac{1}{\lambda}\sum_{i=1}^{r}\sup_{I'_i} \left(f\circ g_i\right)\sup_{\gamma\in\Gamma}\frac{1}{{\rm vf}(\gamma)}\#\{j\in B_i\}.
$$
Now let $\#\{j\in B_i\}$ be $m\leq n$. For this case ${\rm vf}(\gamma)$ must be larger than $m$ since $I'_i$ is a convex closed Jordan curve by assumption (i). Thus we have
\begin{eqnarray}
\eqref{eq3-1}\leq \frac{1}{\lambda}\sup_{\gamma\in\Gamma}\frac{1}{{\rm vf}(\gamma)}\sum_{i=1}^{r}{\rm cvar}(f\circ g_i,\gamma,I'_i) +\frac{1}{\lambda}\sum_{i=1}^{r}\sup_{I'_i} \left(f\circ g_i\right).\label{eq5}
\end{eqnarray}
Using Lemma \ref{lemcvar2},
\begin{eqnarray}
\sum_{i=1}^{r}{\rm cvar}(f\circ g_i,\gamma,I'_i) &=& \sum_{i=1}^{r}{\rm cvar}(f,g_i\circ\gamma,I_i).\nonumber
\end{eqnarray}
Since $g_i\circ\gamma$ is a curve on $I_i$, we can make a new curve $\widetilde{\gamma}$ on $X$ by connecting all curves $g_i\circ\gamma$ for $i=1,\cdots,r$ whose length becomes minimal. Then, by (v) in Proposition \ref{procvar},
\begin{eqnarray}
\sum_{i=1}^{r}{\rm cvar}(f,g_i\circ\gamma,I_i) \leq \sum_{i=1}^{r}{\rm cvar}(f,\widetilde{\gamma},I_i).\nonumber
\end{eqnarray}
Moreover, since $\{I_i\}_{i=1}^r$ are adjacent, by Proposition \ref{procvar2},
\begin{eqnarray}
\sum_{i=1}^{r}{\rm cvar}(f,\widetilde{\gamma},I_i) = {\rm cvar}(f,\widetilde{\gamma},X).
\end{eqnarray}
Thus, by assumption (iv),
\begin{eqnarray}
\underset{X}{\rm Var}(Pf)  &\leq& \frac{1}{\lambda}\sup_{\gamma\in\Gamma}\frac{{\rm vf}(\widetilde{\gamma})}{{\rm vf}(\gamma)}\frac{1}{{\rm vf}(\widetilde{\gamma})}
{\rm cvar}(f,\widetilde{\gamma},X)+\max_i\frac{r(\hat{C}+1)}{\iint_{I'_i}dxdy}\nonumber\\
&=& \frac{C}{\lambda}\sum_{i=1}^{r}\underset{X}{\rm Var}(f)+L,\label{eq2}
\end{eqnarray}
where
$$
L:= \max_i\frac{r(\hat{C}+1)}{\iint_{I'_i}dxdy}
$$
is independent of $f$. Here we use the same procedure for the second term of Eq.\eqref{eq5} as in the calculations from Eq.\eqref{eq6-1} to Eq.\eqref{eq6-2}.
By assumption (v),
\begin{eqnarray}
\underset{X}{\rm Var}(P^nf) &\leq& \left(\frac{C}{\lambda}\right)^n\underset{X}{\rm Var}(f)+
L\sum_{j=0}^{n-1}\left(\frac{C}{\lambda}\right)^j\label{ineq}\\
&<& \left(\frac{C}{\lambda}\right)^n\underset{X}{\rm Var}(f)+\frac{\lambda L}{\lambda-C},\nonumber
\end{eqnarray}
and therefore, for every $f\in D(X)$ of bounded variation,
\begin{eqnarray}
\lim_{n\to\infty} \sup \underset{X}{\rm Var}(P^nf) < K,\nonumber
\end{eqnarray}
where $K>\lambda L/(\lambda-C)$ is independent of $f$.
 Hence, we define $\mathcal{F}$ by
\begin{eqnarray}
\mathcal{F} = \left\{g\in D(X): \underset{X}{\rm Var}(g) \leq K\right\}.\nonumber
\end{eqnarray}
It is clear that for any density $g\in D(X)$ defined on $X$,
$$
g(x,y)-g(\tilde{x},\tilde{y})\leq \underset{X}{\rm Var}(g)
$$
for any $(x,y),(\tilde{x},\tilde{y})\in X$. Since $g\in D(X)$,  there is some $(\tilde{x},\tilde{y})\in X$ such that $g(\tilde{x},\tilde{y})\leq1$ and then we have $g(x)\leq K+1$.
Thus $\mathcal{F}$ is weakly precompact by the criteria 1 in \cite[page 87]{almcmbk94}.
Moreover, by the criteria 3 in \cite[page 87]{almcmbk94}, a set of functions $\mathcal{F}$ is weakly precompact if and only if:  (a) There is an $M<\infty$ such that $\lVert f\rVert_{L^1}\leq M$ for all $f\in\mathcal{F}$; and (b) For every $\varepsilon>0$ there is a $\delta>0$ such that
$$
\int_A\lvert f(x)\rvert d\mu(x)<\varepsilon \quad\quad
\text{if $\mu(A)<\delta$ and $f\in\mathcal{F}$.}
$$
This implies that there is a $\delta>0$ such that
$$
\int_E P^nf(x) d\mu(x)<\varepsilon \quad\quad
\text{if $\mu(E)<\delta$ and $f\in\mathcal{F}$}
$$
which shows $P$ is smoothing and thus asymptotically periodic by Theorem \ref{thm:spec-decomp}.

\qed

\begin{corollary}\label{maincoro}
Let $S:X\to X$ be a transformation and $P$ be the Frobenius-Perron operator corresponding to $S$. If there exists a number $N\in\mathbb{N}$ such that $S^N$ satisfies Conditions (i)-(v) in Theorem \ref{maintheorem}, then, for all $f\in D(X)$, $\{P^nf\}$ is asymptotically periodic.
\end{corollary}

\begin{proof}

By assumption, we find $\{P^{nN}f\}$ is asymptotically periodic for any $f\in D$. Thus one can find a period $\tau<r!$ such that $P^{\tau N}$ is exact. Moreover, we immediately see that $P^{\tau N}$ has an invariant density given by \eqref{ap-sd}. Therefore, by Proposition 5.4 in \cite{toyokawa2020sigma}, $P$ is constrictive and thus asymptotically periodic.
\end{proof}

\section{Two-dimensional example}\label{sec:example}

In this section we offer a new two dimensional example illustrating our results. % Theorem \ref{thm:spec-decomp} and Theorem \ref{maintheorem}.

For parameters $\alpha\in\mathbb{R}$ and $\beta\in (1,2]$, consider  the two-dimensional transformation $S:\mathbb{R}^2\to\mathbb{R}^2$ given by
\begin{eqnarray}\label{2dim}
S(x,y)=(y, \alpha y+ T(x) ),\ \ \ {\rm with}\ \ \ T(x)=\begin{cases}
\beta x+\beta+1 & (x<0)\\
-\beta x +\beta+1 & (x\geq 0).
\end{cases}
\end{eqnarray}
Here the transformation $T$ is the generalized tent map, a straightforward modification of \eqref{eq:hat map}.  As we noted previously, $T$ is statistically periodic \cite{provatas91a} and more precisely, the Frobenius-Perron operator corresponding to $T$ has period $2^n$ when the parameter $\beta$ satisfies
\begin{eqnarray}
2^{1/2^{n+1}}<\beta\leq 2^{1/2^n}\ \ \ {\rm for}\ \ \ n=0,1,2,\cdots.\nonumber
\end{eqnarray}
Next we introduce the transformation $\tilde{S}:\mathbb{R}^2\to\mathbb{R}^2$ defined by
\begin{eqnarray}\label{irynatype}
\tilde{S}(x,y)=\begin{cases}
(\alpha x+ y+1, \beta x) & (x<0)\\
(\alpha x+ y+1, -\beta x) & (x\geq 0)
\end{cases}.
\end{eqnarray}
Then, $S$ and $\tilde{S}$ are homeomorphic, i.e. $\tilde{S}\circ h=h\circ S$ holds where
\begin{eqnarray}
h(x,y)=\begin{cases}
(\frac{y}{\beta+1},\frac{\beta x}{\beta+1}) & (x<0)\\
(\frac{y}{\beta+1},\frac{-\beta x}{\beta+1}) & (x\geq 0)
\end{cases}.
\end{eqnarray}

\begin{remark}

The general system \eqref{irynatype} was also considered by Sushko \cite{iryna2008center}, and they noted a border-collision bifurcation \cite{nusse1992border} in the system.
Although a well-known system similar to Eq.\eqref{irynatype} is the Lozi \cite{lozi1978attracteur}  map given by
$$
S_{\rm Lozi}(x,y) =(1-\alpha|x|+y,\beta x),
$$
the model we treat is different.
Note that if the term $-\alpha |x|$ is replaced by $-\alpha x^2$, we obtain the H\'{e}non \cite{henon1976two} map.

%STOPPED EDITING HERE THURSDAY EVENING

Elhadj \cite{elhadj2010new} suggested a similar example as a new two dimensional piecewise linear chaotic map, noting that \eqref{2dim} can also be written in the alternate form %as a one time-delayed map
\begin{eqnarray}\label{1delay}
x_{n+1}=\alpha x_{n}+T(x_{n-1}).
\end{eqnarray}
Indeed, taking a new variable $X_{n}=(x_n,x_{n-1})$, we can write
\begin{eqnarray}
X_{n+1}=(x_{n},x_{n+1})=(x_{n},\alpha x_{n}+T(x_{n-1}))=
\begin{pmatrix}
0 & 1 \\
\alpha & T(\cdot)
\end{pmatrix}
\begin{pmatrix}
x_{n-1} \\
x_{n}
\end{pmatrix}=
\begin{pmatrix}
0 & 1 \\
\alpha & T(\cdot)
\end{pmatrix}
X_{n},\nonumber
\end{eqnarray}
so that the two-dimensional dynamical system $S$ with $X_{n+1}=S(X_n)$ can be represented by (\ref{2dim}).

If we consider the $d$-time delay difference equation, we can construct a $d$-dimensional discrete dynamical system. Losson \cite{losson95} considered a coupled map lattice which induces a high dimensional map to approximate solutions of differential delay equations. They found periodic orbits of an initial point and a periodicity for the evolution of densities analogous to asymptotic periodicity.

\end{remark}

\subsection{Numerical results}

In this section, we numerically study the transformation \eqref{irynatype} to illustrate our results.

Let $P$ be the Frobenius-Perron  operator corresponding to $\tilde{S}$.
In Figure \ref{fig1} (for positive $\alpha$) and Figure \ref{fig2} (for negative $\alpha$)), we show the {\it support} of $\{P^{500} f_0\}$ for an initial density $f_0=1_{[-5,5]\times[-5,5]}$, $\beta=1.1$. and various values of $\alpha$.
 We see there are  disjoint regions, in Figure \ref{fig1} (a),(e),(g),(h),(i),(k) and in Figure \ref{fig2} (a),(d),(f),(g), and they are the signature of  asymptotic periodicity. For example, in Figure \ref{fig1}h there are five disjoint regions: all points in one region are mapped to another region by $\tilde{S}$ and eventually come back to the initial region by $\tilde{S}^5$. Therefore, the two-dimensional map \eqref{irynatype} has many different periods.
 Conversely, the cases in which there is only one component (e.g. Figure \ref{fig1} (b),(c),...) display asymptotic stability, that is  asymptotic periodicity with $r=1$ .
%}

\begin{figure}[tbp]
\begin{center}
 \includegraphics[width=12cm,bb=0mm 20mm 220mm 180mm,clip]{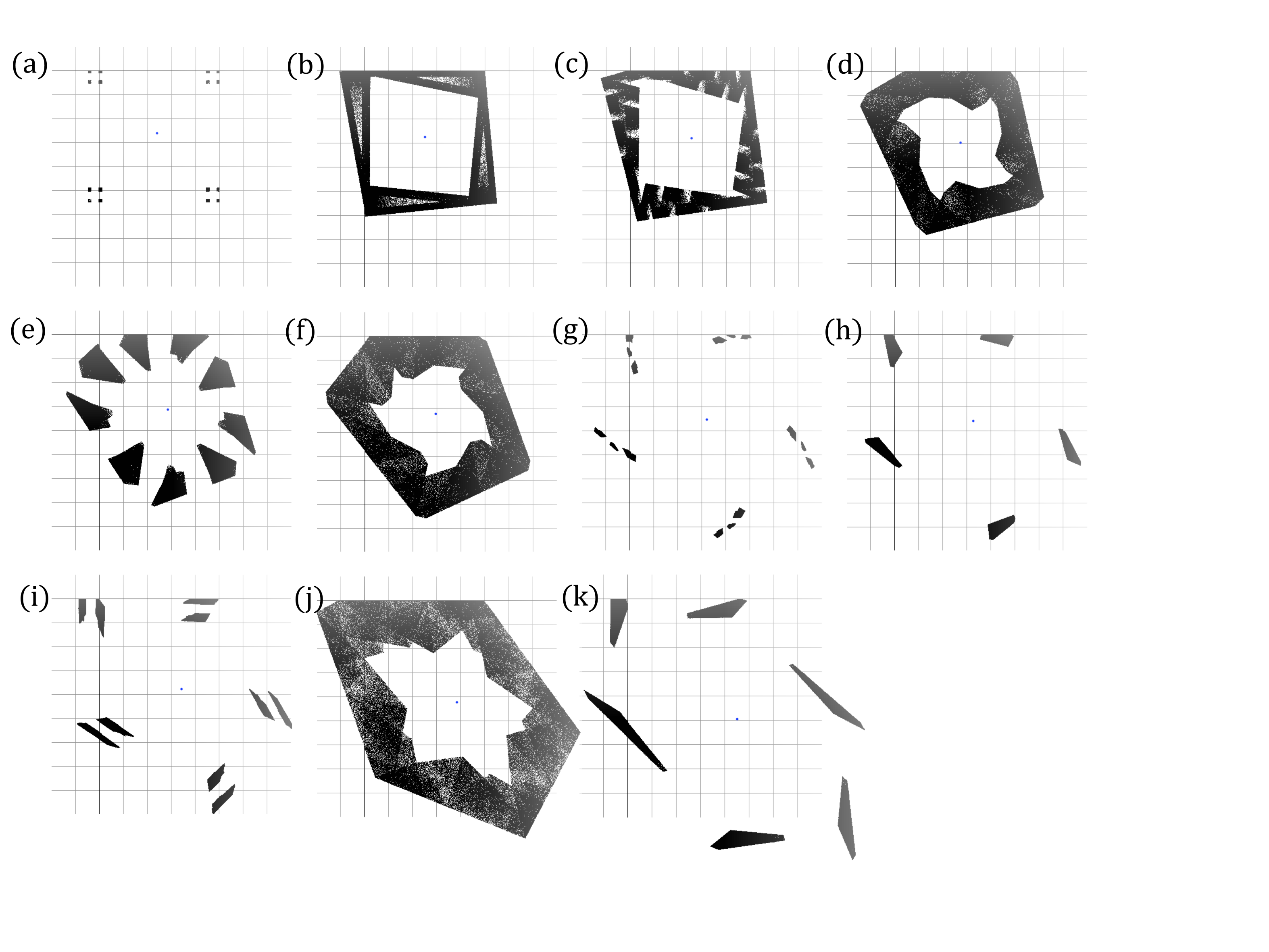}
 \end{center}
 \caption{Numerical illustration of asymptotic periodicity in \eqref{irynatype}.  We show the support of $\{P^{500} f_0\}$ for an initial density $f_0=1_{[-5,5]\times[-5,5]}$, approximated by $1,000\times1,000$ initial points uniformly distributed on $[-5,5]\times[-5,5]$   and various values of $\alpha$ with $\beta=1.1$.
  (a) $\alpha=0.0$,  Period $= 16$; (b) $\alpha=0.1$, Period $= 1$; (c) $\alpha=0.14$, Period $= 1$; (d) $\alpha=0.25$, Period $ =1$; (e) $\alpha=0.34$, Period $= 9$; (f) $\alpha=0.4$, Period $= 1$; (g) $\alpha=0.54$, Period $=12$; (h) $\alpha=0.57$, Period $= 5$; (i) $\alpha=0.64$, Period $= 10$; (j) $\alpha=0.8$, Period $=1$; (k) $\alpha=0.99$, Period $=6$. }
 \label{fig1}
\end{figure}

\begin{figure}[htbp]
\begin{center}
 \includegraphics[width=12cm,bb=0mm 0mm 220mm 100mm,clip]{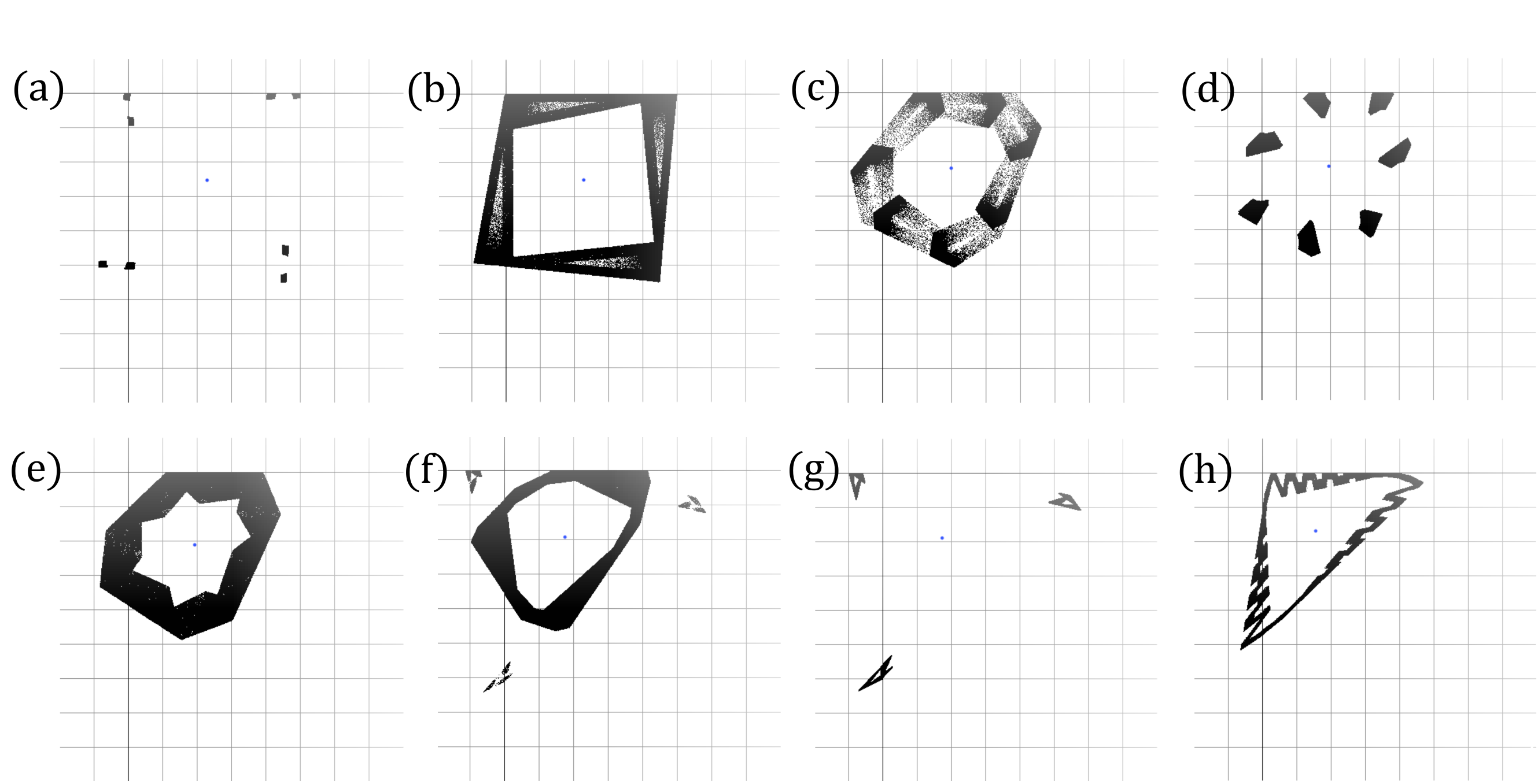}
 \end{center}
 \caption{As in Figure \ref{fig1} with $\beta=1.1$. (a) $\alpha=-0.08$, Period $=8$; (b) $\alpha=-0.1$, Period $=1$; (c) $\alpha=-0.41$, Period $=1$; (d) $\alpha=-0.46$, Period $=7$; (e) $\alpha=-0.5$, Period $=1$; (f) $\alpha=-0.75$, Period $=3$; (g) $\alpha=-0.8$, Period $=3$; (h) $\alpha=-1.14$, Period $=1$;. }
 \label{fig2}
\end{figure}

For smaller $\beta=1.02$ in Figure \ref{fig3} we observe higher periods (Period: (a) 13, (c) 35, (e) 22, (g) 31). In addition to this, we find period $9$ when $\alpha=0.35$. These numerical values of the  periods may be related to a Farey series, see Section \ref{ss:period}.

\begin{figure}[htbp]
\begin{center}
 \includegraphics[width=12cm,bb=0mm 0mm 220mm 100mm,clip]{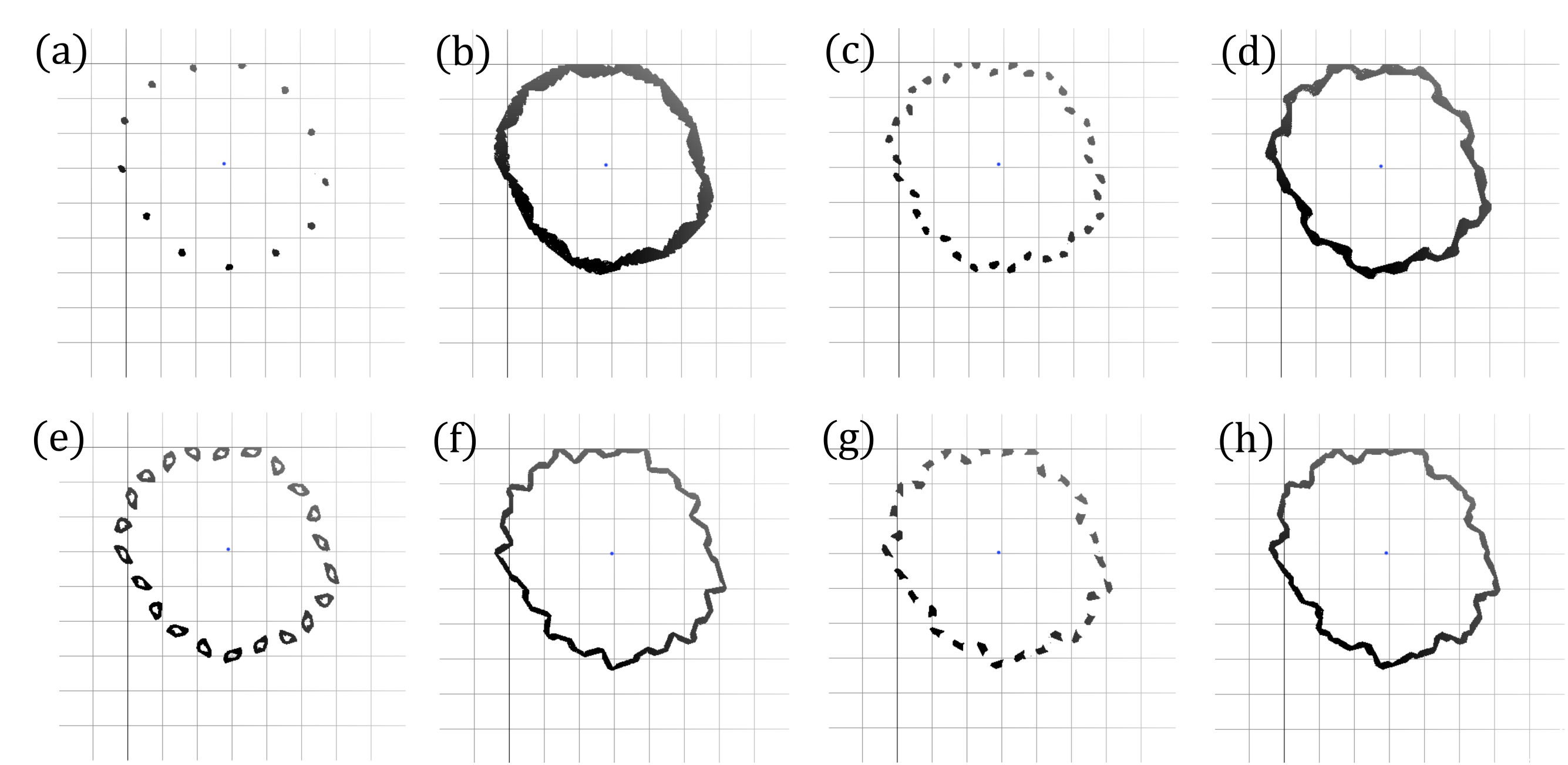}
 \end{center}
 \caption{As in Figure \ref{fig1} with $\beta=1.02$. (a) $\alpha=0.24$, Period $=13$; (b) $\alpha=0.25$, Period $=1$; (c) $\alpha=0.27$, Period $=35$; (d) $\alpha=0.28$, Period $=1$; (e) $\alpha=0.284$, Period $=22$; (f) $\alpha=0.3$, Period $=1$; (g) $\alpha=0.3015$, Period $=31$; (h) $\alpha=0.31$, Period $=1$;.  }
 \label{fig3}
\end{figure}

\subsection{Discussion: Asymptotic periodicity}

Consider \eqref{irynatype} in the context of Corollary \ref{maincoro}. Since \eqref{irynatype} is piecewise linear and the Jacobian $\lambda_n$ for $\tilde{S}^n$ is $\beta^n$, the assumptions (i)-(iv) of Theorem \ref{maintheorem} are satisfied. Thus we need only show the condition (v) holds, that is, $\frac{5}{2\beta^N}r_n<1$ where $r_n$ denotes the number of partitions for $\tilde{S}^n$.

Without loss of generality, it is enough to consider the system (\ref{irynatype}) on the half plane $\mathbb{R}_{\{y\leq 0\}}$ since all points are in  $\mathbb{R}_{\{y\leq 0\}}$ after iterating once.  Let $L$, $M$ and $R$ be the sets $L=\{(x,y)\in\mathbb{R}^2\ |\ x<0,y<0\}$, $M=\{(x,y)\in\mathbb{R}^2\ |\ x=0,y<0\}$ and $R=\{(x,y)\in\mathbb{R}^2\ |\ x>0,y<0\}$ respectively, and denote
\begin{eqnarray}
S_L(x,y)=(\alpha x+ y+1, \beta x) \ \  (x<0),\ \ \ \ S_R(x,y)=(\alpha x+ y+1, -\beta x) \ \  (x\geq 0).
\end{eqnarray}
One immediately has the following properties for (\ref{irynatype}):
\begin{itemize}
\item If $\alpha+\beta>1$, there exists a fixed point $(x_L^{\ast},y_L^{\ast})=(\frac{1}{1-\alpha-\beta},\frac{\beta}{1-\alpha-\beta})\in L$.
\item If $\alpha-\beta<1$, there exists a fixed point $(x_R^{\ast},y_R^{\ast})=(\frac{1}{1-\alpha+\beta},\frac{-\beta}{1-\alpha+\beta})\in R$.
\item The eigenvalues of the Jacobian are $\lambda_L^{\pm}=\frac{\alpha\pm\sqrt{\alpha^2+4\beta}}{2}$ and $\lambda_R^{\pm}=\frac{\alpha\pm\sqrt{\alpha^2-4\beta}}{2}$, and the corresponding eigenvectors are $(\lambda_L^{\pm},\beta)$ and $(\lambda_R^{\pm},-\beta)$.
\item Since $\alpha^2+4\beta>0$ always holds, $\lambda_L^+>1$ if $\alpha+\beta>1$. This implies the fixed point $x_L^{\star}$ is unstable.
\item If $\alpha^2\geq4\beta$ and $\alpha>2$, then $\lambda_R^{+}>1$ and $x_R^{\ast}$ is an unstable node, and  almost all points diverge in this case.
\item In the case $\alpha^2<4\beta$, then $\lambda_R^{\pm}$ is complex which implies $x_R^{\ast}$ is an unstable focus (if $\alpha>0$), a center (if $\alpha=0$) and a stable focus (if $\alpha<0$).
\end{itemize}

Based on these observations, we focus on parameters satisfying $\alpha+\beta>1$, $\alpha-\beta<1$, $\alpha^2<4\beta$, $\alpha>0$ and $1<\beta\leq 2$. Note that although asymptotic periodicity is observed even when $\alpha<0$ (Figure \ref{fig2}), here we assume $\alpha>0$ to simplify the arguments.
%For each parameters, we estimate a dissipative part of the system.

First, we know the saddle point $x_L^{\star}\in L$. Let $D_0$ be the set
\begin{eqnarray}
D_0:=\{(x,y)\in L\cup M\ |\ y-y_L^{\ast}<\frac{\beta}{\lambda_L^-}(x-x_L^{\ast})\}.
\end{eqnarray}
From the instability of the fixed point $(x_L^{\ast},y_L^{\ast})$, one can immediately conclude that all points in $D_0$ eventually diverge. Next, let $c$ be a $y$-intercept of the line $y-y_L^{\ast}<\frac{\beta}{\lambda_L^-}(x-x_L^{\ast})$, that is, $c=\beta x_L^{\ast}(1-1/\lambda_L^-)$. Then the $x$-intercept of the line can be calculated as $S_L(0,c)=c+1$ which is always negative when $\alpha+\beta>1$.

Second, consider the inverse sets $D_i:=S_R^{-i}(D_0)\cap \mathbb{R}_{\{y\leq 0\}}$ and the inverse of a point $(0,c)$, $S_R^{-i}(0,c)$ for $i=1,2,3,\cdots$. Note that all points in $D_i$ for some $i$ diverge. Now let $\ell$ be a minimum number  $i$ such that the $y$-coordinate of $S_R^{-i}(0,c)$ is positive. Then let $C$ be a set defined by
\begin{eqnarray}
C:=\mathbb{R}_{\{y\leq 0\}}\backslash\bigcup_{i=0}^{\ell}{D_i}.
\end{eqnarray}
Then $C$ becomes the candidate for the attracting region.
Figure \ref{fig4} illustrates the partition of the half plane ($y\leq0$) and regions $C$ and $\{D_i\}_{i=0}^\ell$ for the case $\ell=5$.

Third, let $p$ (and $q$) be the $x$-coordinate of the intersection point of the line $y=0$ and the line generated by $S_R^{-\ell}(0,c)$ and $S_R^{-(\ell-1)}(0,c)$ (and $S_R^{-\ell}(0,c)$ and $(x_R^*,y_R^*)$). We can consider three cases depending on the values of $p,q$ relative to $1$.

Figure \ref{fig5} illustrates the three possible cases. Figure \ref{fig5} (a) shows the case in which both $p, q < 1$, (b) shows the case $p < 1 < q$, and (c) shows the case with $1 < p,q$. We immediately observe that points in $C$ may leave from $C$ in case (a) because of the black region, and if $q\geq 1,$ then $C$ is a conserved region. Therefore, we can construct a dynamical system which acts on a bounded set by giving the restricted system $\tilde{S}:C\to C$ for the case (b) or (c).
We focus on case (b).

From these observations, there exists a partition $I_0,\cdots,I_{\ell+1}$ in $C$ such that $\tilde{S}^{\ell+2}(I_i)\subset I_i$ for any $i=0,\cdots,\ell+1$ (see Figure \ref{fig6}).
The condition (iv) implies the ratio of entry point of the before and after curve by the inverse transformation. In our case, an increase of the number of entry points happens only for the $S|_{I_{\ell+1}}^{-1}$, in other words, the other $S|_{I_i}^{-1}$, $i=0,\cdots,\ell$ does not increase the entry points because of the rotational behavior. Since $\beta>1$, for some $t$, $\beta^t$ which is the Jacobian of $\tilde{S}^t$, might be larger than $C$. Namely, the condition (iv) and (v) in Theorem \ref{maintheorem} would be satisfied for $\tilde{S}^{t}$ with sufficiently large $t$.

However, it is difficult to check the condition (iv)   because of the impossibility to calculate the change of entry points for all curves. Thus, we do not have checkable sufficient conditions for the assumption (iv) to prove asymptotic periodicity for $\tilde{S}$, which is strongly suggested by the numerical results.
 In the case (c), although it is more complicated due to the black region, we may use similar arguments after one more iterate $S^{-(\ell+1)}_R(0,c)$.

\begin{figure}[tbp]
\begin{center}
 \includegraphics[width=8cm,bb=0mm 0mm 70mm 50mm,clip]{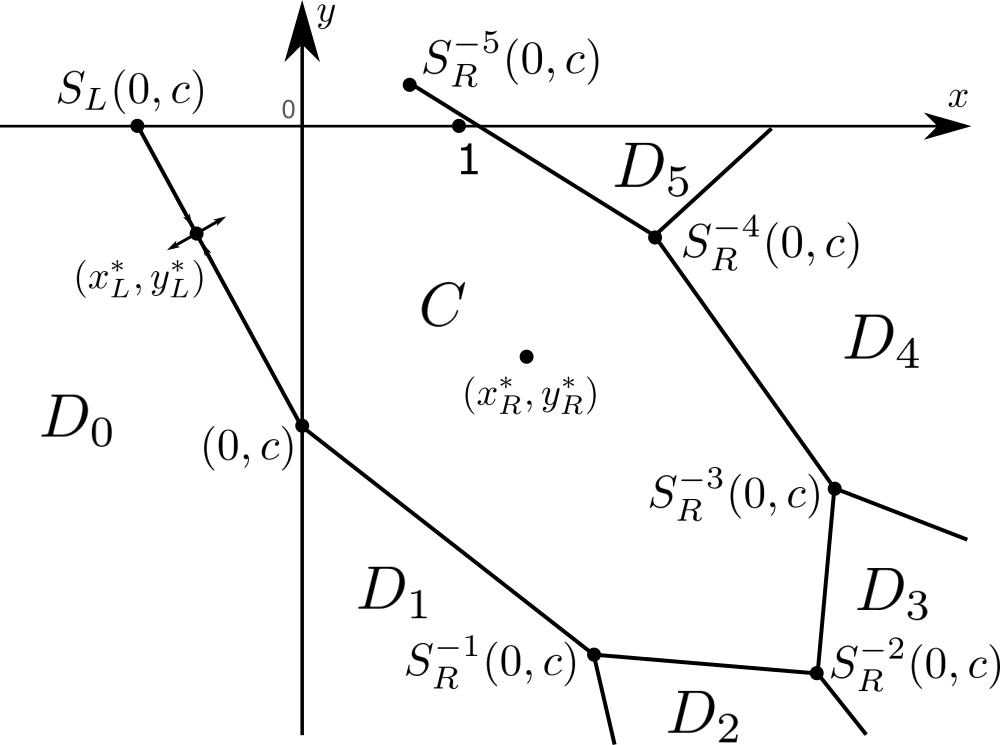}
 \end{center}
 \caption{The regions $D_i$, $i=0,1,\cdots,5$, and $C$ are illustrated when $\ell=5$. The fixed point $(x_L^*,y_L^*)$ is a saddle and $(x_R^*,y_R^*)$ is an unstable focus,}
 \label{fig4}
\end{figure}

\begin{figure}[tbp]
\begin{center}
 \includegraphics[width=12cm,bb=0mm 0mm 200mm 50mm,clip]{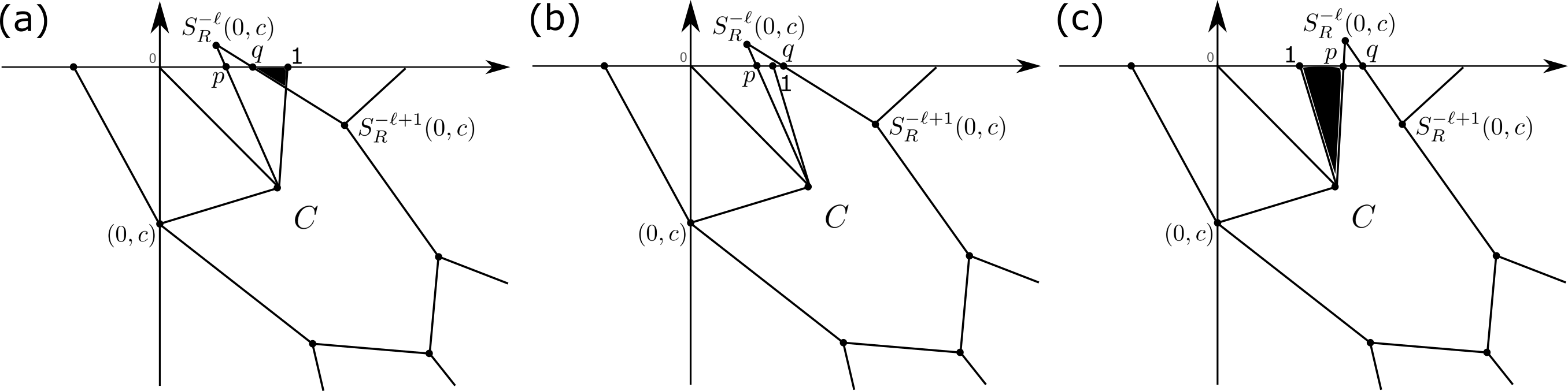}
 \end{center}
 \caption{The situation can be separated into three cases depending on positions of $p,q$ and $1$. (a) the case $p,1<1$, (b) the case $p < 1 < q$, and (c) the case $1<p,q$.}
 \label{fig5}
\end{figure}

\begin{figure}[tbp]
\begin{center}
 \includegraphics[width=12cm,bb=0mm 0mm 130mm 50mm,clip]{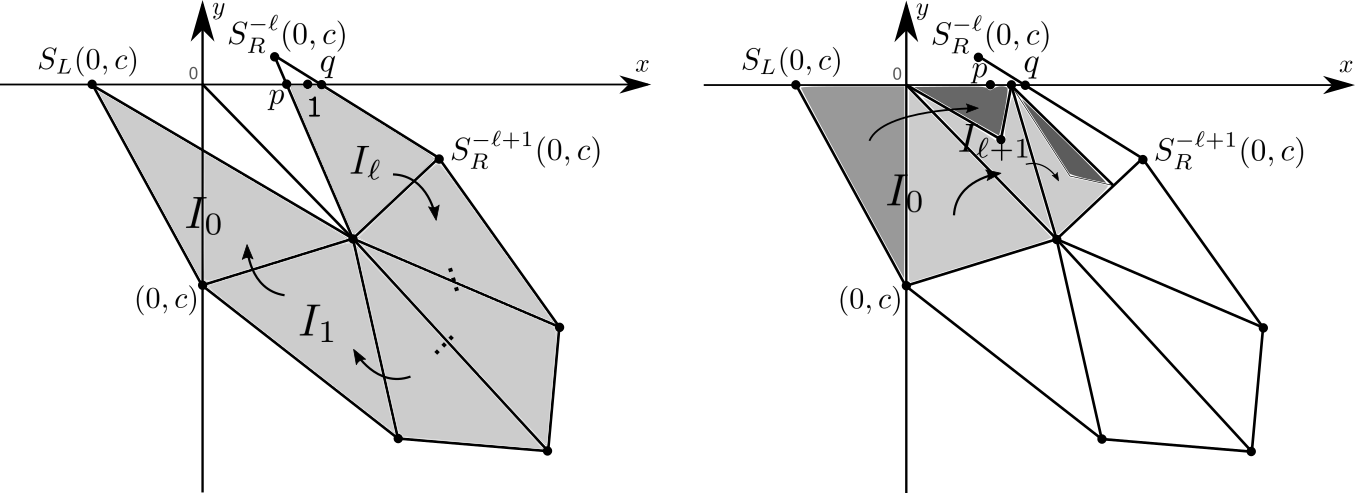}
 \end{center}
 \caption{Illustrations of the result of iterating the regions $\{I_i\}_{i=0}^{\ell+1}$ by $\tilde{S}$.}
 \label{fig6}
\end{figure}

Finally, we will estimate the parameter conditions such that  $q\geq1$ since at least $q$ must be larger than 1 to be a conservative system. If we set $S_R^{-1}({\bm x})=A{\bm x}+{\bm b}$,
then $S_R^{-n}({\bm x})=A^n{\bm x}+(A^{n-1}+\cdots+A+I){\bm b}$ where
$$
A=\begin{pmatrix}
0 & -1/\beta \\
1 & \alpha/\beta
\end{pmatrix},\ \ \
\bm{b}=\begin{pmatrix}
0  \\
-1
\end{pmatrix}.
$$
Thus we have %From a basic scheme of the linear algebra, one can obtain
$$
A^n=\frac{1}{\nu_--\nu_+}\begin{pmatrix}
\nu_+^n\nu_- - \nu_+\nu_-^n & (-\nu_+^n + \nu_-^n )/\beta \\
(\nu_+^{n+1}\nu_- - \nu_+\nu_-^{n+1})\beta & -\nu_+^{n+1} + \nu_-^{n+1}
\end{pmatrix},
$$
where $\nu_{\pm}$ are eigenvalues of $A$ with $\nu_{\pm}=\frac{\alpha}{2\beta}\pm\frac{\sqrt{4\beta-\alpha^2}}{2\beta}i\in\mathbb{C}$.
By using the above equations, we may write $S^{-n}(0,c)$, $p$ and $q$ explicitly. However,
 not only is the calculation complicated, but also we cannot obtain the number $\ell$ for each set of parameters. Thus we numerically show only approximate values of $\alpha$ which gives the condition for $q\geq 1$ for some values of $\beta$ in Table \ref{t:table11}.

\begin{table}[htb]
\begin{center}
  \begin{tabular}{|c|c|c||c|c|c||c|c|c||c|c|c|} \hline
    $\beta$ & $\ell$ & $\alpha<$ & $\beta$ & $\ell$ & $\alpha<$ & $\beta$ & $\ell$ & $\alpha<$ & $\beta$ & $\ell$ & $\alpha<$ \\ \hline \hline
    1.01 & 14 & 1.85664 & 1.06 & 7 & 1.57519 & 1.2 & 4 & 1.15624 & 1.7 & 3 & 0.53436\\
    1.02 & 11 & 1.78516 & 1.07 & 7 & 1.56379 & 1.3 & 3 & 1.03992 & 1.8 & 2 & 0.32593\\
    1.03 & 9  & 1.71214 & 1.08 & 6 & 1.48766 & 1.4 & 3 & 0.78308 & 1.9 & 2 & 0.13439\\
    1.04 & 8  & 1.65753 & 1.09 & 6 & 1.46841 & 1.5 & 3 & 0.66496 & 2.0 & 2 & 0.00000\\
    1.05 & 8  & 1.64245 & 1.1  & 5 & 1.45765 & 1.6 & 3 & 0.58999 & & &  \\\hline
  \end{tabular}
   \caption{\label{t:table11}For each $\beta$, the value $\alpha$ which gives the condition for $q\geq 1$ are calculated numerically.}
  \end{center}

\end{table}

\subsection{Discussion: Period}\label{ss:period}

We would like to be able to predict the period of the asymptotic periodicity in \eqref{irynatype} for a given set of parameters $(\alpha,\beta)$, but although we can make a partition $\{I_i\}_{i=0}^{\ell+1}$ as the previous section, we cannot find the period or the relation between $\ell$ and period.

The numerical results in Figure \ref{fig1},\ref{fig2} and \ref{fig3} tantalizingly remind one of the Farey series\footnote{The definition of Farey series of order $n$, denoted by $F_n$, is the set of reduced fractions in the closed interval $[0,1]$ with denominators $\leq n$, listed in
increasing order of magnitude. For instance, $F_1=\{0,1\}$, $F_2=\{0,1/2,1\}$, $F_3=\{0,1/3,1/2,2/3,1\}$ and so on. (See \cite{apostol1976modular} for details). One of the important properties of Farey series is that each fraction in $F_{n+1}$ which is not in $F_n$ is the mediant of a pair of consecutive fractions in $F_n$. For example, $2/5$ in $F_5$ is made by $1/3$ and $1/2$ in $F_4$, that is, $1/3\oplus 1/2=(1+1)/(3+2)=2/5$. The operation $\oplus$ is called the Farey sum.}.
In dynamical systems, periodic structures based on the Farey series sometimes appear, for instance in circle map models of cardiac arrhythmias \cite{boyland1986,glass1983,swiatek1988}.
The fraction $l/n$ corresponds to a rotation number of the system, that is, every periodic orbit has period $n$. Nakamura \cite{nakamura2017} proved that the Markov operator corresponding to the perturbed piecewise linear map \eqref{eq:NS} exhibits asymptotically periodicity, and clarified the relationship of the periods associated with the Farey series for various parameters.

For our example \eqref{irynatype}, Figure \ref{fig3} displays asymptotic periodicity with period 22 in between values of the  parameter $\alpha$ giving rise to period 13 and 9, while period 35 is between 13 and 22, and period 31 is between 22 and 9.  Moreover, we observe period $58$  $ (\alpha=0.322)$, $76$ $ (\alpha=0.328)$ and $47$ $ (\alpha=0.42)$. That is, there exist parameters for which the system has asymptotic periodicity with period $p_1+p_2$ between the parameters which give periods $p_1$ and $p_2$. To take this observation and relate the periodicity of  \eqref{irynatype} to the Farey series is a matter for future research.

\section*{Acknowledgement}
The work is supported by the Natural Sciences and
Engineering Research Council (NSERC) of Canada, and the Ministry of Education, Culture, Sports, Science and Technology
through Program for Leading Graduate Schools (Hokkaido University ``Ambitious
Leader's Program'').

\end{document}